\documentclass{article}
\usepackage{amsmath, amsthm, amssymb}
\usepackage{array}
\usepackage{pdflscape}
\usepackage{natbib}
\usepackage{lscape}
\usepackage{tabularx}
\usepackage{afterpage}
\newcolumntype{L}{>{\centering\arraybackslash}m{3cm}}

\newcommand{\sbar}{\bar{s}}

\newcommand{\Sone}{\mathcal{U}_1}
\newcommand{\Stwo}{\mathcal{U}_2}
\newcommand{\Stau}{\mathcal{U}_\tau}
\newcommand{\Sv}{\mathcal{U}_v}
\newcommand{\Svprime}{\mathcal{U}_{v^\prime}}

\usepackage{color,fancybox,graphicx}
\usepackage{enumerate}
\usepackage{hyperref}
\usepackage{mathtools}
\newtheorem{remk}{Remark}[section]

\newtheorem{prop}{Proposition}[section]

\newtheorem{thm}{Theorem}[section]

\begin{document}
\newpage
\markboth{Staff Scheduling at Airports}{A Resource Management Problem}

\title{\huge \textbf{\Large An Efficient Algorithm to the Integrated Shift and Task Scheduling Problem}\\
{\normalsize \textbf{G S R Murthy}\footnote{\emph{Corresponding author}:G S R Murthy (murthygsr@gmail.com)}~~~~~~~~ \ \textbf{T R Lalita}\footnote{trlalita@gmail.com} }\\
{\scriptsize \emph{SQC \& OR Unit, Indian Statistical Institute, St.No.8, Habsiguda,
Hyderabad 500007, India}}}


\maketitle

\begin{abstract}
%
This paper deals with operational models for integrated shift and task
scheduling problem. Staff scheduling problem is a special case of this with
staff requirements as given input to the problem. Both problems become hard
to solve when the problems are considered with flexible shifts. Current
literature on these problems leaves good scope for potential research. In
this article, we propose a new method to solve the integrated problem and its
special case, the staff scheduling problem. We consider these problems with
wide flexibility - a feature that is addressed in a limited way in the
existing literature. We introduce a new technique to solve the problem with
large demand efficiently. When the objective function is the number of
workers, we provide a tight lower bound that is easily computable. Through a
number of numerical experiments with live and simulated problem instances, we
demonstrate huge savings in the solution times over the existing ones.

\end{abstract}

 \thispagestyle{empty}


\newpage

\section{Introduction}\label{sec:introduction}
Personnel scheduling problems arise in a variety of applications and deal with
assignment of shifts to workforce over a planning horizon. A large number of
applications involve flexible work schedules. Workforce requirements over the planning
horizon are induced by task characteristics such as duration, number of workers
required to perform the task, deadlines, etc. Demand, the number of workers
required, in each time period of the planning horizon, may be known exactly or
assumed according to a predictable pattern, is necessary for the purpose of
planning. The flexibility in staff (or work) schedules
has two components: (i) type of shift which
specifies duration, breaks and their positioning within shift, etc., (ii) time gap between successive
shifts, bounds on the number of shifts in the planning horizon, bounds on the total
number of worker-hours, days-off, etc. The constraints in the latter component, part of the
tour scheduling process,
are imposed due to labour laws, company regulations, employees preferences and so on.
The complexities and challenges are aggravated
by the flexibilities of staff and task schedules. One of the factors that is
ignored in much of the existing literature is not including breaks within shifts
(see \cite{thompson2007scheduling}). Breaks can be included using implicit formulations
(see \cite{sungur2017shift},\cite{aykin1996optimal}), but this would dramatically increase the
size of the problem.

In this article, we are concerned with two versions of personnel scheduling problem over
a discrete planning horizon. In the first version, staff schedules are flexible but the
tasks are fixed and the demand of resources (number of personnel required) for each time
period of the planning horizon is specified. The second version is an extension of the
first and it allows tasks to be scheduled within specified time periods and the tasks may
have precedence relationships. The workforce demand is a result of task scheduling. The
objective in both versions is to minimize the number of workers or an associated cost.
The second version is referred to as \emph{integrated shift and task scheduling problem}
(ISTSP). The problem is so complex that it calls for special formulations and methods for
solving it. \cite{stolletz2010operational} computes the possible tours in a further
restricted case of first version of the problem (shifts without breaks, shifts restricted
to 4~am to 9~pm) to the tune of $10^{19}$. ISTSP is intractable for exact solution
approaches. A common mathematical programming approach to solving ISTSP uses set covering
formulation or its variants (\cite{dantzig1954letter}). Solution approaches presented in
\cite{maenhout2016exact} and \cite{volland2017column} are some of the recent
contributions in this direction.

To the best of our knowledge, the methods for staff scheduling or ISTSP in the existing
literature have not considered a wide range of problems.  For example,
\cite{stolletz2010operational} and \cite{brunner2014stabilized} have considered
discontinuous tour scheduling problems and not the problems with continuous demand. While
the former considers shifts without breaks, the latter considers shifts with only one
break. Similarly, \cite{volland2017column} does not consider breaks within shifts.
Moreover, these articles implicitly express that problems with larger demands (by
classifying them under small, medium and large) are harder to solve. Against this
backdrop, we believe that this article makes an important and significant contribution.
The main contribution of this article is that we provide a new method for ISTSP that can
\begin{itemize}
  \item reduce solution times drastically,
  \item solve problems with large demands in approximately the same time taken for problems with
        small demands, and
  \item handle wide flexibility in shifts resulting from multiple breaks.
\end{itemize}
The organisation of the rest of this article is as follows. In the next section, we start
with the genesis of this work and present a brief discussion on the extensions of the
model assumptions and their consequences. This will be followed
by a brief literature review with focus on recent contributions relevant to this
paper.  In Section~\ref{sec:problemdescription}, we present the
problem description, our formulations, solution approach and a discussion on their applications.
Section~\ref{sec:liveinstances} describes our numerical experiments with data
from live problems and simulation. The simulation exercises are carefully planned so
as to compare our approach with existing methods. Section~\ref{sec:summarysection}
presents the summary of the experimental results. The article is concluded in
Section~\ref{sec:conclusion}, with a summary and possible scope for future research.

\section{Motivation and Literature Review} \label{sec:motivation}
This work is an extension of a problem that we received from a software company.
For ease of cross referencing, we shall call this the Software Industry Problem (SIP) in this article.
The requirement was to develop a method for determining staff schedules with
flexible shifts to meet workforce demand specified for every 30-minute time period~(TP) over
one week planning horizon (336 TPs) with an objective of minimizing the number
of workers. Demand for a selected week is shown in Fig.~\ref{fig:Callcenterdemand}.
The admissible shifts in this problem should
satisfy four conditions: (i)~shift has two tea breaks each of 15 minutes duration
and one lunch break of 60 minutes, (ii)~no break in the first 90 minutes,
(iii)~at least 90 minutes gap between any two successive breaks, and (iv)~the
duration of the shift including breaks is 9 hours.
\begin{figure}[h]
  \centering
  \includegraphics[width=11cm]{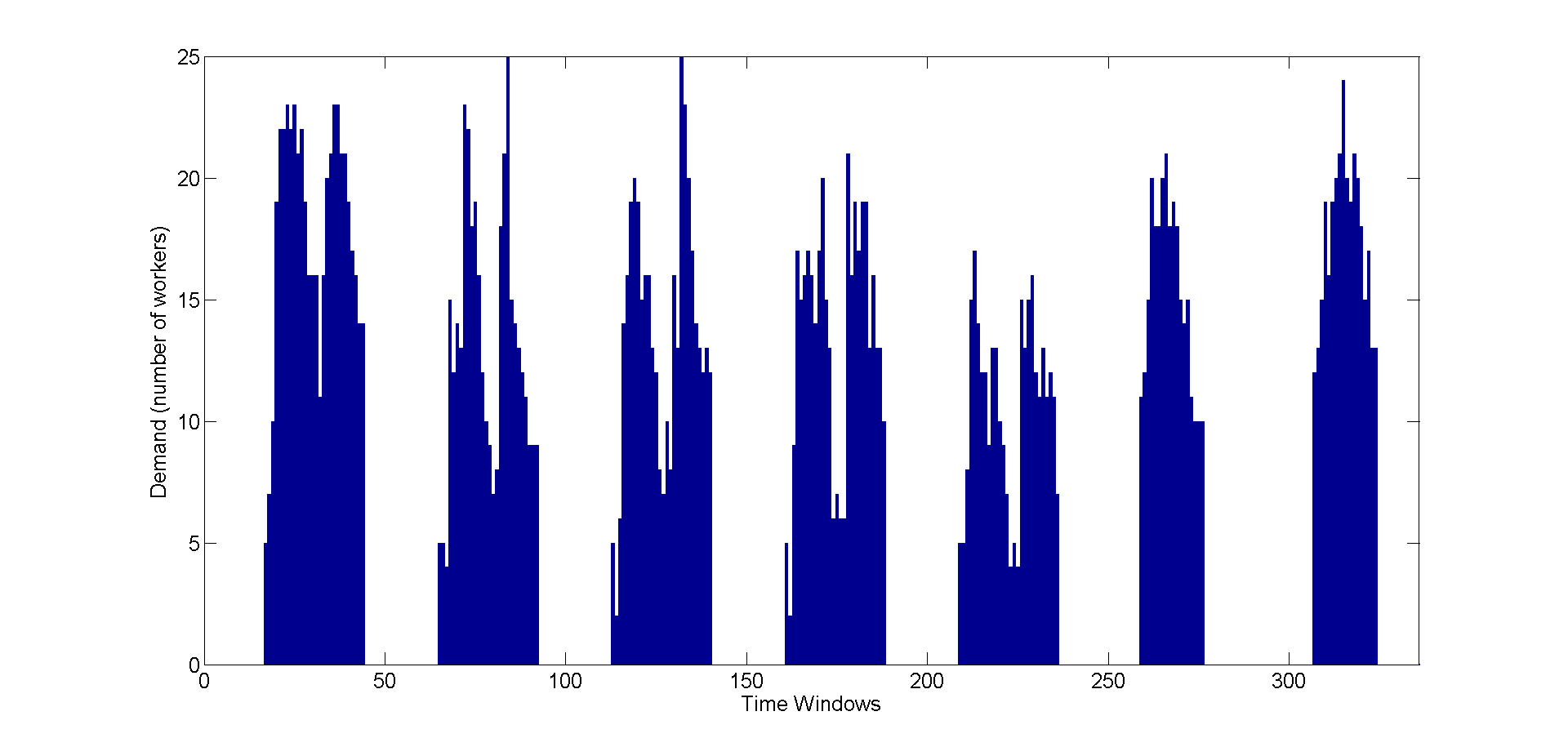}
  \caption{SIP demand for every 30-minute over one week (336 TPs). No demand
            during 10~pm to 8~am. Total demand 1258 worker-hours.}\label{fig:Callcenterdemand}
\end{figure}

This work is the outcome of our effort to solve the SIP in its full
flexibility. Encouraged by the results and the nature of our approach, we noticed
that it can be extended to ISTSP.

There is vast literature on personnel scheduling problem. The problem has been classified
into different categories depending upon the areas of applications, models and solution
approaches. For a detailed review on personnel scheduling problems,
see \cite{DBLP:journals/eor/ErnstJKS04} and \cite{van2013personnel},  and references therein.
Different approaches are pursued for solving staff scheduling problems
(see \cite{alfares2004survey}, \cite{bellenguez2007branch} and \cite{brunner2010midterm}).
The main hurdle in solving staff scheduling problems is their size. In SIP, there are  260
different shifts satisfying the stated conditions. Since a
shift can commence at the beginning of any of the TPs, there are $12480 (=48\times
260)$ possible shift schedules within a day.  High scheduling flexibility results in huge
number of personnel schedules. Mathematical programming formulations
for solving staff scheduling problem are mostly based on the set-covering
formulation of \cite{dantzig1954letter}. As the set covering formulation requires the
all personnel schedules, decomposition and column generation techniques through implicit formulations are
commonly used to handle the situation. Implicit formulations are developed for
several applications (see \cite{thompson2007scheduling}, \cite{thompson1995improved},
\cite{jarrah1994solving}, \cite{jacobs1996overlapping},\cite{aykin1996optimal},
\cite{jacobs1996overlapping} and \cite{brunner2009flexible}, and
\cite{sungur2017shift}). Yet, the problem remains complex as the implicit
formulations often result in large number of constraints (see
\cite{bellenguez2007branch} and \cite{brunner2009flexible}). Decomposition
technique is used to breakdown the problem into stages so as to reduce the size of the problem (see \cite{jarrah1994solving},
\cite{alfares2004survey}, \cite{stolletz2010operational} and \cite{brunner2014stabilized}).
Also see \cite{brucker2011personnel} for a discussion on models and complexities in
personnel scheduling problems.

The work in this article is closely related to three articles:
(i)~\cite{stolletz2010operational}, (ii)~\cite{brunner2014stabilized} and
(iii)~\cite{volland2017column}. The first two of these  deal with staff scheduling with
given resource input, and the the third one deals with ISTSP. First, we shall brief the
contributions of these articles and other works related to them.

\cite{stolletz2010operational} introduced a \emph{reduced set covering
formulation} to solve a personnel touring problem  for check-in systems at
airports. His model considers a fortnightly planning horizon comprising
30-minute TPs. The staff requirements are needed only in the TPs confined to
time between 4~am to 9~pm. The restrictions on the shifts are that they must
start and end between 4~am and 9~pm,  no breaks are allowed and their
durations must be between 6~TPs to 20~TPs (i.e., between 3 hours to 10
hours). With these restrictions, there are 330 staff schedules; and using
these, the problem was solved through a binary integer programming
formulation. \cite{brunner2014stabilized} expanded the scope of the problem
by incorporating one period lunch break in the shifts. They report poor
convergence of the column generation subroutine and introduce stabilized column generation procedure.
SIP is similar to the one considered by \cite{brunner2014stabilized} but with
higher complexity as it involves multiple and more flexible breaks in the
shifts ($12480$ personnel schedules per day). Though SIP is also
discontinuous (i.e., workforce is required only between 8~am to 10~pm), we
considered the more general problem of continuous case, that is, staff
requirements may be there in all TPs. Therefore,  our model is more general
and more complex, in terms of the size of the problem, compared to that of
\cite{brunner2014stabilized}. \cite{stolletz2014rolling} develop a rolling
planning horizon-based heuristic for the tour scheduling problem for agents
with multiple skills and flexible contracts in check-in counters at airports.

When supply vector is fixed, ISTSP reduces to the well known resource constrained project
scheduling problem (RCPSP) with personnel as resources. See \cite{hartmann2010survey} for
a survey on RCPSP and its extensions. The published literature on ISTSP is limited. For
applications of the problem see \cite{belien2008branch}, \cite{maenhout2013integrated},
\cite{di2014assessment}, \cite{kim2015two}, \cite{volland2017column} in health sector;
\cite{belien2013integrated} for scheduling problem in an aircraft maintenance company;
and \cite{bassett2000assigning} for a scheduling problem in an agro-based industry. On
the solution methods for the problem, see \cite{alfares1997integrated},
\cite{bailey1995optimization},  \cite{bassett2000assigning}, \cite{belien2008branch} and
\cite{belien2013integrated} for some early papers on the subject.
\cite{maenhout2016exact} decompose the problem into a master problem and a personnel
scheduling subproblem. The personnel schedules used in the restricted master problem are
generated iteratively through the personnel scheduling subproblems. Thus, the approach
comprises decomposition and column generation techniques. In their model, the TPs are
days, and therefore, shifts within days are not considered.

\cite{volland2017column} propose an ILP formulation (referred to as MIP in their article)
for ISTSP with a weekly planning horizon and develop a column generation method to derive
a good starting feasible solution with a lower bound for solving the MIP. The method uses
implicit formulations for two subproblems - the shift scheduling subproblem (S-SP) and
the task scheduling subproblem (T-SP). The two subproblems are linked to a restricted LP
relaxation of the MIP to generate personnel and task schedules. The process is continued
iteratively by augmenting the restricted master problem with newly generated personnel
and task schedules until the optimum objective value of the LP relaxation is attained.
Let FLPR stand for the final LP relaxation. After building the (personnel and task
schedule) columns of FLPR, they drop a set of task columns (by retaining only a selected
set of \emph{high quality} task columns) from FLPR, and add additional personnel schedule
columns to it if possible, and solve it as an ILP. Taking the optimal solution of this
ILP as a warm start, they solve the MIP.

Our model for ISTSP and the approach to solve it differ from those of
\cite{volland2017column} in three ways: (i)~breaks within shifts are more flexible
(\cite{volland2017column} does not incorporate breaks), (ii)~we do not use column
generation approach, and (iii)~we do not solve the MIP which is more complex. To get a
solution for ISTSP, we decompose it into two ILP subproblems. Solving the two subproblems
produces an optimal solution if the objective function depends only on the \emph{shift
patterns} and their positioning, and near optimal solutions if the objective is to
minimize the number of workers. The decomposition scheme in our model is based on shift
patterns. All shift patterns (allowing full admissible flexibility) can be listed using a
simple computer program instead of deriving them through a complex traditional approach
of using implicit ILP formulations. Further, we provide a lower bound for the number of
workers when there is an upper limit on the number shifts per worker.

\section{Problem description and Formulation} \label{sec:problemdescription}
In this paper, we consider ISTSP over a cyclic planning horizon of one week split
into $T$ TPs of equal duration of length $\omega$ minutes ($\omega=15$ or $30$ are
considered for the instances of this paper). In staff scheduling, a personnel
schedule assigns shifts to a worker over the planning horizon fulfilling work
schedule restrictions. Each personnel schedule will yield a binary vector in
$\mathbb{R}^{T}$ with 1s representing availability of the worker, who is assigned the
schedule, in the respective TPs. Sum of all assigned personnel schedule vectors is a
nonnegative integer vector (the supply vector), and its $j^{th}$ coordinate
specifies the number of available workers in TP $j$. On the other hand, task
scheduling involves determining start TP of each task satisfying precedence
relationships. This will yield a non-negative integer vector in $\mathbb{R}^T$ (the demand
vector) specifying the number of workers required in each TP. Under the considered
ISTSP, the problem is to determine the personnel schedules (to be assigned to
workers) and a task schedule so that the resulting supply vector is greater than or
equal to the resulting demand vector. The objective is to minimize the number
of assigned personnel schedules or sum of their given associated costs.
See Table~\ref{tab:notation} for notation and input parameters.

The planning horizon is $\mathcal{T}=[1,2,\ldots,T]$. Given $K$ tasks,
numbered 1 through $K$, task $k$ has the following inputs: (i) start window $[l_k,\
u_k]$ in which the task must start, where $l_k,\ u_k \in \mathcal{T}$ with
$l_k \leq u_k$, (ii) $d_k$, duration of the task specified as the number of TPs, and (iii)
the resource vector $\mathbf{r}_k=(r_{k1},r_{k2},\ldots,r_{kd_k})$, where $r_{kj}$
is the number of workers required in the $j^{th}$ TP of task $k$,
$j=1,2,\ldots,d_k$. For the precedence relationships among tasks, the input is a set
of task pairs $\mathcal{P}$. If $(k,\ k^\prime)\in \mathcal{P}$, it means task~$k$
should precede task~$k^\prime$. We use the notation $k\prec k^\prime$ to imply that
$(k,\ k^\prime)\in \mathcal{P}$.
\afterpage{%
{\small
\begin{table}[h]
  \centering
    \caption{Notation}\label{tab:notation}
   \begin{tabular}{p{4cm}p{8cm}}
     \hline
      \textbf{Indices} & \\
     $k$ &  task number \\
     $j$ &  time period (TP) number in the planning horizon  \\
     $i$ & shift pattern index, $i=1,2,\ldots,q$ \\
     $v$ & shift schedule index, $v=1,2,\ldots,\tau$ \\
     $u$ & worker index, $u=1,2,\ldots,w$, where $w$ is maximum
           number of workers \\
      &  \\
     \textbf{Parameters} & \\
     $T$ & number of time periods in the planning horizon \\
     $K$ & Number of tasks \\
     $q$ & number of shift patterns\\
     $l_k$ & earliest start period of task $k$ \\
     $u_k$ & latest start period of task $k$ \\
     $d_k$ & duration of task $k$ in number of TPs \\
     $\tau$ & number of shift schedules from stage~1, $=\sum_{ij}x_{ij}$\\
     $\mathbf{r}_k=(r_{k1},r_{k2},\ldots,r_{kd_k})$ & demand vector of task $k$,
            where $r_{kj}$ is the number of workers required in the $j^{th}$ TP of task $k$ \\
     $SL_{min}/SL_{max}$ & minimum/maximum limits on the length of a shift \\
     $SG_{min}$ & minimum gap (in number of TPs) to be maintained between two successive shifts \\
     $\mathbf{s}=(s_1,s_2,\ldots,s_m)$ & shift pattern of length $m$ TPs, $s_1,\ldots,s_m$ are
                  worker availabilities in the appropriate TPs \\
     $(\mathbf{s}^i,\ j)$ & shift schedule, shift pattern $\mathbf{s}^i$ starting at TP $j$ \\
       &  \\
     \textbf{Sets and vectors} &  \\
     $\mathcal{T}=[1,2,\ldots,T]$ & $\mathcal{T}$ is the planning horizon and  $T$ is the number of TPs \\
     $[l_k,\ u_k]$ & start time window of task $k$ \\
     $\mathcal{S}=\{\mathbf{s}^1,\ldots,\mathbf{s}^q\}$ & set of shift patterns\\
     $\mathcal{P}$ & set of task pairs, $(k,\ k^\prime)\in \mathcal{P}$ means
                     $k\prec k^\prime$, that is, task $k$  must be completed before starting task $k^\prime$  \\
     $\mathbf{R}=(R_1,R_2,\ldots,R_T)^t$ & the demand vector, $R_j=$ the number of workers required in TP $j$ \\
     $\mathbf{S}=(S_1,S_2,\ldots,S_T)^t$ & the supply vector, $S_j=$ the number of workers available in TP $j$ \\
     $\Sone,\ \Stwo,\ \ldots,\ \Stau$ & assigned shift schedules from stage~1 arranged in the ascending order of their
                    start TPs \\
      &  \\
     \textbf{Variables} & \\
     $y_{kj}$ & indicator variable which is one if task $k$ is assigned to TP $j$ \\
     $x_{ij}$ & number of shift schedules $(\mathbf{s}^i,\ j)$ (determined in stage~1 and assigned to $x_{ij}$ workers in stage~2) \\
     $z_{uv}$ & indicator variable, $=1$ if $\mathcal{U}_v$ is assigned to worker $u$\\
     \hline
   \end{tabular}
\end{table}
}

\clearpage
}

For the staff scheduling, the following inputs/flexibility types are considered:
(i)~shifts with or without breaks (as specified) having length between a specified
minimum ($SL_{min}$) and a maximum ($SL_{max}$), (ii)~shift start window is the range of
TPs within a day during which a shift can start, (iii)~gap between any two successive
shifts assigned to a worker in terms of number of TPs must be greater than or equal to a
specified lower limit $SG_{min}$, and (iv) an upper limit either on the number of shifts
or total hours assigned to any worker in the week. Note that (ii) above is pertinent to
certain specific instances. For example, the 330 shifts referred to in
\cite{stolletz2010operational} must start between 4~am and 6:30~pm but it depends on the
shift length as well; the start window for a 3-hour shift is 4~am to 6:30~pm, and start
window for a 3.5~hour length shift is 4~am to 6~pm, and so on. Even in the case of SIP,
no shift can start from 10~pm to 8~am (from Fig.~\ref{fig:Callcenterdemand} it can be
observed that there is no demand during this period).

The traditional approach to handle ISTSP with flexible schedules is to use implicit
formulations and iterative methods using column generation techniques. In
\cite{volland2017column}, a staff schedule is implicitly formulated for the entire
planning horizon combining shifts and their assignment. In order to mitigate the
complexity, \cite{stolletz2010operational} used a reduced set covering formulation where
predetermined daily shifts are implicitly embedded in the planning horizon. In this
paper, we reduce the complexity further. We present a two-stage approach to solve this
problem directly without using implicit formulations for shift patterns, iterative
procedures and the column generation techniques. We achieve this by using shift patterns
as the key to the entire planning. We first define shift pattern formally.

\noindent \textbf{What is a shift Pattern?}\label{shiftpatterndef} \\
A shift pattern of length $m$ is a binary $m$-vector that satisfies all the
shift constraints such as $SL_{min}\leq m \leq SL_{max}$ and the shift break period
rules. We shall denote a shift pattern by $\mathbf{s}=(s_1,s_2,\ldots,s_m)$.

\noindent \textbf{What is a shift schedule?}\\
A \emph{shift schedule}, denoted by $(\mathbf{s}, t)$,  is a combination of a
shift pattern $\mathbf{s}$ and a TP $t$. A shift schedule is used to specify
that a worker who is assigned $(\mathbf{s}, t)$ must start a fresh shift at
TP $t$ and work according to shift pattern $\mathbf{s}$. The $t$ in
$(\mathbf{s}, t)$ may be specified relative to a day (in this case, $t$
ranges from 1 to 48 with $\omega=30$) or relative to the entire planning
horizon (in this case, $t$ ranges from 1 to 336 with $\omega=30$). In our
models in this paper, $t$ is relative to the entire planning horizon.

\cite{stolletz2010operational} used shift schedules relative to day and generated 330 of
them. The shift patterns (embedded in his shift schedules) have only 1s as their
coordinates (as no break periods are considered) and their lengths vary from 6 to 20.
Similarly, in the model used by \cite{volland2017column}, there are 25 underlying shift
patterns containing only 1s as their coordinates (as no break periods are considered).
For SIP, we have 260 shifts patterns because we consider tea and lunch breaks. Each of
these patterns can be described using shift patterns of length 18 with exactly fourteen
1s, two consecutive 0s and two 0.5s. For example,
$\mathbf{s}=(1,1,1,1,0,0,1,1,1,1,0.5,1,1,0.5,1,1,1,1)$. The two 0s ($s_5$ and $s_6$)
stand for a lunch break and the two 0.5s ($s_{11}$ and $s_{14}$) stand for the two tea
breaks\footnote[1]{Firstly, we are abusing the definition of shift pattern by allowing
the fraction 0.5. This is only done to handle breaks of half TP. This will not cause any
hinderance in solving the problems using methods of this paper. Next, it might appear to
violate the condition that the gap between two successive breaks must be at least 90
minutes. Note that this can still be upheld by allowing the first tea break in the first
15 minutes of the corresponding TP and allowing the 2nd tea break in the last 15 minutes
of the corresponding TP.}. The numbers 0, 0.5 and 1 are the proportions of a TP that a
worker is available.  It must be noted that these shift patterns can be generated
implicitly through ILP formulations but that becomes very complicated. Instead, we can
use a simple computer program to generate all the shift patterns effortlessly as we did
for this problem.

We are now ready to present our two-stage solution method for ISTSP. The basic idea is
that we first determine the shift schedules in stage~1, and assign them to workers in
stage~2. The stage~1 problem is described in Section~\ref{sec:stageone} and stage~2 in
Section~\ref{sec:stagetwo}.

\subsection{Shift Pattern Subproblem - Stage~1} \label{sec:stageone}
In this stage we consider two sets of decision variables. The first set of decision
variables assigns the TPs to task starting times. The second set of variables decide
the number of shift patterns assigned to TPs so as to meet the required workforce demands.
These decisions yield the supply of workforce in each TP, and the two sets of decision
variables are linked through supply-demand constraints. The objective function of the
problem will be taken as the cost of shifts.

Let $\mathcal{S}=\{\mathbf{s}^i:\ i=1,2,\ldots,q\}$
be the set of all shift patterns and let $m_i$ be the length of
$\mathbf{s}^i,\ i=1,2,\ldots,q$. Let $y_{kj}$ be 1 if task $k$ starts in TP
$j$, and equal to 0 otherwise. Let $x_{ij}$ be the number of shift schedules
$(\mathbf{s}^i, j)$, $i=1,2,\ldots,q$ and $j\in \mathcal{T}$.

The task assignment $Y=(y_{kj})$ induces a demand vector $\mathbf{R}=(R_1,R_2,\ldots,R_T)^t$,
where $R_j$ is the number of workers required in TP $j$. The expression
for $R_j$ is given by
\begin{equation}\label{eq:RjDefn}
     R_j = \sum_{k=1}^{K} \sum_{i=1}^{d_k} r_{ki} y_{k\theta(j-i+1)},
\end{equation}
where $\theta(\cdot)$ is the \emph{wrap function} for the cyclic time horizon, that is,
$\theta(0)=T,\ \theta(-1)=T-1,\ldots, $ and $\theta(T+1)=1,\ \theta(T+2)=2,\ldots$.

Similarly, $X=(x_{ij})$ induces a supply vector $\mathbf{S}=(S_1,S_2,\ldots,S_T)^t$, where
$S_j$ is the number of workers available in TP $j$. The expression for $S_j$ is given by
\begin{equation}\label{eq:SjDefn}
      S_j = \sum_{i=1}^{q} \sum_{t=1}^{m_i} s_{t}^{i} x_{i\theta(j-t+1)},
\end{equation}
where $\mathbf{s}^i=(s_{1}^{i},\ldots,s_{m_i}^{i})$  and $\theta(\cdot)$ is the wrap function defined above.

Let $c_{ij}$ be the cost of shift schedule $(\mathbf{s}^i, j)$.
Then our \textbf{stage~1 problem} for ISTSP,
is given by
\begin{align} \label{obj:stage1}
 \texttt{ Minimize~~ }  \sum_{i=1}^{q}\sum_{j=1}^{T} c_{ij}x_{ij}~~~~~~~~~~~~~~~~~~~~~~~~~~~~~~~~~ & \\
 \texttt{ subject to~~~~~~~~~~~~~~~~~~~~~~~~~~~~~~~ } &\nonumber \\ \label{con:SgreaterthanR}
              \sum_{i=1}^{q} \sum_{t=1}^{m_i} s_{t}^{i} x_{i\theta(j-t+1)}
               \geq \sum_{k=1}^{K} \sum_{i=1}^{d_k} r_{ki} y_{k\theta(j-i+1)}, & \texttt{ for }j=1,2,\ldots, T,\\ \label{con:precedences}
                \sum_{j=1}^{T}(j+d_k-1)y_{kj} \leq \sum_{j=1}^{T}jy_{k^\prime j}  \texttt{ for all } & (k,\ k^\prime)\in \mathcal{P},  \\  \label{con:assigntask1}
            \sum_{j=1}^{T}y_{kj}  = 1 \texttt{ for }k=1,2,\ldots,K,~~~~~~~~~~~~~  & \\  \label{con:assigntask2}
             \sum_{j=1}^{l_k-1}y_{kj} + \sum_{j=u_k+1}^{T} y_{kj}  = 0,~\texttt{ for all } k~~~~~  & \\  \label{con:assigntask2b}
                            y_{kj}\in \{0,\ 1\}   \texttt{ for all }i,\ j,~~~~~~~~~~~~~~~~~~~ &  \\
              x_{ij}\texttt{s are nonnegative integers for all } i, j. & \label{MPNonnegativity}
\end{align}
Above, (\ref{con:SgreaterthanR}) is the supply-demand constraints, (\ref{con:precedences})
takes care of the precedence relationships,
(\ref{con:assigntask1}) and (\ref{con:assigntask2}) ensure that all tasks start in their
designated start windows $[l_k,\ u_k]$\footnote[2]{In the actual implementation
of the model for solving the problem, $y_{kj}$s will be defined
only for $l_k\leq j \leq u_k$.}. The objective function is the
total cost of assigned shifts.

\begin{remk}
If we change the objective function of (\ref{obj:stage1}) to
$\sum_{i=1}^{q} \sum_{t=1}^{m_i} s_{t}^{i} x_{i\theta(j-t+1)}$, the total supply, and
minimize it, then we will be minimizing the over cover (=total supply minus total demand)
because the total demand is a constant that does not depend on task scheduling.
\end{remk}

\subsection{Staff Assignment Problem - Stage~2} \label{sec:stagetwo}
From stage~1 solution, we have the shift schedules that will meet the staff demand
requirements satisfying the shift constraints. We now assign these shift schedules
to workers, maintaining staff scheduling constraints involving minimum/maximum number of
shifts/hours per worker, days-off per worker, etc. The stage~2 formulation requires
preparation of inputs. This process will be described first.

From stage~1 output, collect all shift schedules $(\mathbf{s}^i, j)$ for which
$x_{ij}>0$ and sort them according to the ascending order of $j$.
The number of such shift schedules is $\tau = \sum_{ij} x_{ij}$. Let
$\Sone=(s^{i_1}, j_1),\ \Stwo=(s^{i_2}, j_2),\ \ldots,\ \Stau=(s^{i_\tau}, j_\tau)$ be the
$\tau$ shift schedules. Note that $j_1\leq j_2\leq \ldots \leq j_\tau$, and a shift schedule
$(\mathbf{s}^i,\ j)$ with corresponding $x_{ij}$ is repeated $x_{ij}$ times in the list.

Choose a large positive integer $w$ representing maximum number of workers
available for scheduling during the planning horizon. Label the workers as
$1, 2, \ldots, w$.  Define the  decision variables of
stage~2 as follows: $z_{uv}=1$ if worker $u$ is assigned shift schedule $\Sv$,
$z_{uv}=0$ otherwise, $u=1,2,\ldots,w,\ v=1,2,\ldots,\tau$.

In order to meet the supply-demand constraints, we must assign each of the $\tau$
shift schedules to workers. Note that $f(v)=\sum_{u=1}^{w}uz_{uv}$ is worker label
to which shift schedule $v$ is assigned to. Therefore, stage~2 objective function is
$\max_v f(v)$ which we minimize to minimize the number of workers. This objective function
is linearized by introducing a dummy variable $\xi$ and a constraint as follows.
\begin{align}\label{obj:stage2}
  \texttt{Minimize~~ }   & \xi \\
  \texttt{subject to } & \nonumber \\
                       & \sum_{u=1}^{w}uz_{uv} \leq \xi,\ \  v=1,2,\ldots,\tau. \label{obj:stage2binding}
\end{align}
Next, we formulate the constraints of stage~2 problem.
\begin{description}
  \item[Assignment Constraints:] Each of $\Sone$ to $\Stau$ must be assigned to workers. This translates to
          \begin{equation}\label{con:stage2assign}
            \sum_{u=1}^{w} z_{uv} =1 \texttt{ for } v=1,2,\ldots,\tau.
          \end{equation}
  \item[Maximum number of shifts:] Suppose a worker can have at most $b$ shifts in the planning
                 horizon. This translates to
                 \begin{equation}\label{con:stage2maxnoOfshifts}
                         \sum_{v=1}^{\tau} z_{uv} \leq b, \texttt{ for } u=1, 2, \ldots, w.
                 \end{equation}
  \item[Rest period:] Rest period, the gap between any two successive shifts assigned to a worker, must
                 be at least $g$ TPs. For this, we first define a overlapping pair of shift schedules
                 $\Sv$ and $\Svprime$. For $v< v^\prime$, say that $\Sv$ and $\Svprime$ are
                 overlapping if $j_{v^\prime} \leq j_v + m_{i_v} -1 + g$. Call $(v, v^\prime)$ an overlapping
                 pair if $\Sv$ and $\Svprime$ are overlapping, $1\leq v < v^\prime \leq \tau$. To ensure rest period,
                 any worker can be assigned at most one of $\Sv$ and $\Svprime$ if $(v, v^\prime)$
                 is a overlapping pair. This translates to
                 \begin{equation}\label{con:stage2restperiod}
                         z_{uv} + z_{uv^\prime} \leq 1 \texttt{ for  every } u \texttt{ and every overlapping pair }(v, v^\prime).
                 \end{equation}
  \item[Total Hours:] The total time of any worker must not exceed $H$ TPs. Note that
                 $\sum_{v=1}^{\tau}m_{i_v}z_{uv}$ is the total duration, in number of
                 TPs, of worker $u$ over the planning
                 horizon. Therefore, the constraints are
                 \begin{equation}\label{con:stage2maxhours}
                         \sum_{v=1}^{\tau}m_{i_v}z_{uv}\leq H,\ \ \ u=1, 2, \ldots, w.
                 \end{equation}
   \item[Days-Off:] We shall assume that day-off must start from the first TP of a day.
                 In a 5-day week, a worker must get two consecutive days off. We shall
                 formulate the constraints taking $\omega=30$ and one week planning horizon (we can imitate
                 the same for other values of $\omega$). Constraints for a 6-day week
                 can be derived in a similar fashion. A two-day period can be represented
                 by the shift pattern $\bar{\mathbf{s}}$ with $\sbar_i=1$ for $i=1,2,\ldots,96$.
                 Introduce dummy shift schedules $\Sv=(\bar{\mathbf{s}}, j_v),\ v=\tau+1,\ldots,\tau+7$,
                 where $j_1,j_2,\ldots,j_7$ are the starting TPs of days 1 to 7 respectively
                 (i.e., $j_1=1,\ j_2=49,\ j_3=97$, and so on). With an abuse of convention,
                 we shall interpret the dummy shift schedules as days-off. That is,
                 $z_{u(\tau + 1)}=1$ will be interpreted as worker $u$ having first two days of
                 the week off. Interpret $z_{u(\tau + 2)}=1$ as 2nd and 3rd days of the week off,
                 and so on. With this, the two-days-off constraints for worker $u$, $u=1, 2, \ldots, w$,
                 can be written as
                    \begin{align}\label{con:stage2daysoff1}
                         \sum_{v=\tau + 1}^{\tau + 7} z_{uv} = 1, \texttt{ and for every overlapping pair} (v, v^\prime) \\
                        z_{uv} + z_{uv^\prime} \leq 1, \texttt{ where } 1 \leq v \leq \tau,\ \ \ \tau+1\leq v^\prime \leq \tau +7. \label{con:stage2daysoff2}
                 \end{align}
\end{description}
Thus, constraints for stage~2 problem can be picked from
(\ref{con:stage2assign}) to (\ref{con:stage2daysoff2}) depending upon the context, and
perhaps can be augmented with some more if necessary.

Optimality of solutions obtained by the \textbf{two-stage method}~(TSM) depends upon the nature of
objective function. The following theorems are useful in this regard.
\begin{thm}\label{thm1}
If the objective function of ISTSP is a function of shift schedules, then
the two-stage method produces an optimal solution.
\end{thm}
\begin{proof}
Suppose the objective function of the ISTSP is a function of shift schedules. Assume that
the problem has an optimal solution. Solving the Stage~I of the two-stage method, we
obtain for the original minimization problem, an optimal number of shift schedules and a
minimum cost associated with the shift schedules. As the Stage~II problem minimizes the
number of workers in the organization, it does not affect the total number of shift
schedules to be assigned to the workers, i.e., the Stage~I solution. As the feasible
region for the stage I problem is the same as the feasible region for the ISTSP, the
two-stage method produces an optimal solution.
\end{proof}

\begin{thm} \label{thm2}
If the objective of ISTSP is to minimize the number of workers with one of the
constraints as \ref{con:stage2maxnoOfshifts}, then $\lceil \frac{B}{b} \rceil$ is a
lower bound for the number of workers, where $B$ is any lower bound for stage~1
objective function, the number of shift schedules assigned.
\end{thm}
\begin{proof}
Suppose, the objective of ISTSP is to minimize the total number of workers and the number
of shifts that can be assigned to a worker is limited by $b$ (constraint
~\ref{con:stage2maxnoOfshifts}). The two-stage method in this case may not provide an
optimal solution, however, we can compute a lower bound to the optimal solution. As a
maximum of $b$ shift schedules can be assigned to each worker, assume the best case
scenario. That is, each worker in the organization can be assigned $b$ shift schedules
out of $B$ obtained in Stage I, without overlap. This gives a lower bound, $\lceil
\frac{B}{b} \rceil$ to the number of workers in the organization.
\end{proof}

One of the factors that appears to have a significant bearing on the solution time of ISTSP
is the total demand $\sum_j R_j$. In the existing literature, the problem
instances are classified as small, medium and large based on this factor.
We introduce a \emph{split technique} to handle problems with large demands.
It has a cascading effect on reducing the solution time
of ISTSPs with large demands.

\subsection{The Split Technique} \label{sec:splittechnique}
Consider an ISTSP and suppose that $\mathbf{R}$ is a demand vector that is optimal or near
optimal. We split the demand vector $\mathbf{R}$ into sum of two new demand vectors $\mathbf{R}^1$ and $\mathbf{R}^2$ so that
$\mathbf{R}=\mathbf{R}^1+\mathbf{R}^2$. Then, we solve two new subproblems with fixed demand vectors $\mathbf{R}^1$ and $\mathbf{R}^2$
separately using the two-stage approach and combine the solution to get a solution
to the original problem. If $w^1$ and $w^2$ are optimal (or near optimal) objective
values of the two subproblems, then we have a solution for the ISTSP with $w^1+w^2$
workers. We shall explain this approach with the help of some examples.

One of the problem instances (corresponds to P4 in
Table~\ref{tab:results336}) is a problem with fixed $\mathbf{R}$ (no task scheduling)
and has a total demand 3736 worker-hours. Solving this using two-stage method,
stage~1 was solved to near optimality in 32 seconds with a lower bound of
499; but stage~2 got abandoned due to
insufficient memory. Then, we solved the two subproblems taking
$R_{j}^{1}=\lfloor  \frac{R_j}{2} \rfloor$ and $R_{j}^{2}=R_j-R_{j}^{1},\
j=1,2,\ldots,336$. The resulting subproblems have demands 1836 (for $\mathbf{R}^1$)
and 1900 (for $\mathbf{R}^2$). Solving these two problems using two-stage method
yielded the following results. The $\mathbf{R}^1$-subproblem resulted in a near optimal
solution (in 201 seconds) with 52 workers, and the  $\mathbf{R}^2$-subproblem resulted in a near optimal
solution (in 198 seconds) with 54 workers. Combining the solutions of the
two subproblems, we have a solution to the original problem
with 106 workers. From Theorem~\ref{thm2}, $100$~($=\lceil \frac{499}{5}  \rceil$)
is a lower bound for the problem. Therefore, the solution with 106 workers is at
least $94$\%~($=100-\frac{106-100}{100}\times 100$) optimal, and the problem
is solved in less than 7 minutes.

\noindent \textbf{How to solve faster?}\\
Consider a case where the total demand is so large that even after
splitting the demand vector, we still have a problem. Even for such cases, we
solve only two subproblems to get a solution. Consider the problem with a
total demand of 5136 worker-hours (see P20 in Table~\ref{tab:results336}). For
this problem, we take  $R_{j}^{1}=\lfloor  \frac{R_j}{3} \rfloor$ and
$R_{j}^{2}=R_j-\frac{2R_{j}^{1}}{3},\ j=1,2,\ldots,336$. With this, the
demands for $R^1$ and $R^2$ subproblems are 1669 and 1798 worker-hours
respectively. Note that $\mathbf{R}=2\mathbf{R}^1+\mathbf{R}^2$. Solving the two subproblems, we found a solution for
$\mathbf{R}^1$-subproblem with 57 workers, and for $\mathbf{R}^2$-subproblem with 58 workers.
To obtain a solution for the original problem, we apply the solution of
$\mathbf{R}^1$-subproblem to two sets of 57 workers each, and apply
the solution of $\mathbf{R}^2$-subproblem to another set of 58 workers. The resulting
allocation is a solution to the original problem with 172 ($=2\times 57 +58$)
workers. To find the optimality percentage, we use the lower bound of the
stage~1 problem with original demand vector. For the instance in question,
the stage~1 problem with the original demand vector with demand of 5136
worker-hours produced an optimal solution (in 2 seconds) with 852 shift schedules.
From Theorem~\ref{thm2}, the number of workers is at least
171, and hence the solution obtained using split technique is at least 98.8\%
optimal. The whole process took 6 minutes and 22 seconds.

Consider another instance with a total demand of 5615 worker-hours (P21 in
Table~\ref{tab:results336}). Splitting $\mathbf{R}=2\mathbf{R}^1+\mathbf{R}^2$ with $\mathbf{R}^1=\lfloor
\frac{\mathbf{R}}{3} \rfloor$ and solving this problem took 10 minutes 18 seconds. The number of
workers in this case is 183 and the lower bound from stage~1 solution is 181
(stage~1 took 2 seconds). Taking $\mathbf{R}=3\mathbf{R}^1+\mathbf{R}^2$ with $\mathbf{R}^1=\lfloor \frac{\mathbf{R}}{4}
\rfloor$ and solving this problem (P22) took only 4 minutes 2 seconds. The
number of workers in the resulting solution is 182.  In general, we can use
$\mathbf{R}^1=\lfloor \frac{\mathbf{R}}{\rho} \rfloor$, where the \textbf{\emph{splitting factor}}
$\rho > 1$. Choosing large $\rho$ will reduce the
solution time but will affect the optimality. Therefore, we should choose
$\rho$ judiciously.
\begin{remk} \label{rem:solveonlytwice}
Under the split technique, we solve only two subproblems with demand vectors
$\mathbf{R}^1$ and $\mathbf{R}^2$ and use the solutions to derive a solution to the original problem.
\end{remk}

\begin{remk}\label{rem:readyRvector}
The split technique is found to be very effective in solving problems with large demands.
However, this method requires the demand vector $\mathbf{R}$. For problems of ISTSP, the optimal demand
vector is to be obtained first in order to apply the split technique. It must be noted that
stage~1 of our approach produces optimal or near optimal demand vector $\mathbf{R}$ very efficiently even
for the cases where the demand is very high (see Table~\ref{tab:results336}).
Thus, our two stage approach clubbed with the split technique (if needed) can solve ISTSPs
even with large demands very efficiently.
\end{remk}
Based on our empirical experience, we make the following proposition.
\begin{prop} \label{prop1}
The total demand does not appear to be a factor that affects the complexity of ISTSP.
\end{prop}

\section{Live Instances and Numerical Experiments} \label{sec:liveinstances}
In this section we assess the performance of the two-stage approach clubbed with split
technique (where necessary) with a number of live and simulated instances. For this, we
consider two categories of problems. The first one corresponds to the type of problems
considered in \cite{stolletz2010operational} and \cite{brunner2014stabilized} where the
tasks are already scheduled and we have a demand vector $\mathbf{R}$ as input to the
problem. The second category of problems corresponds to the type of problems dealt with
in \cite{volland2017column} where both tasks and shifts have to be scheduled, that is,
proper ISTSPs. The live instances for the first category are taken from requirements from
software industry, airport check-in counter staff requirements and call center data. For
the second category of problems, we use simulated data. For the purpose of comparison,
the data are simulated using the distributions specified in \cite{volland2017column} as
well as some new distributions.   We also have one live data from emergency medical
services (108 service in India) for this category. For the first category of problems, we
consider $\omega=30$ and $T=336$, and for the second category, $\omega=15$ and $T=672$.
All problems are treated with cyclic planning horizon.

\noindent \textbf{Types of shift patterns}\\
For our numerical experiments, we considered four types of shift patterns
described below. The first three of them are used in problems with
$\omega=30$ and $T=336$. The fourth one, FX29, is
used in problems with $\omega=15$ and $T=672$.
\begin{description}
  \item[FX260] Under this, all shifts have fixed duration of 9 hours (18 TPs of length
        $\omega=30$) with breaks satisfying conditions (i) to (iv) stated at the beginning of
        Section~\ref{sec:motivation}. The number 260 is the number of shift patterns under FX260.
  \item[FL15] There are 15 patterns under this with durations varying from 3 hours to 10 hours.
            Relief breaks are incorporated at appropriate positions depending on duration of the shift
            (3-5h: no break, 5.5-6h: one 15-minute break, 6.5-8h: one 30-minute break,
            8.5-10h: two 15-minute and one 30-minute breaks).
  \item[FL135] \cite{brunner2014stabilized} considered lunch breaks in the shift patterns. The duration
            of the shifts varies from 3 to 10 hours with exactly one 30-minute break with the condition that
            no break in the first one hour and in the last one hour of the shift. There are
            135 such shift patterns.
  \item[FX29] These are shift patterns with durations varying from 3 hours to 10 hours
      without breaks. These are the 29 patterns considered in \cite{volland2017column}
      for their numerical experiments.
\end{description}

In all, we solved 40 instances (see tables \ref{tab:results336} and
\ref{tab:results672}). All problems are solved using the LINGO professional
solver Version 13.0 on an i7 64-bit processor with 2.80GHz clock
speed and 16 GB RAM running on a Windows 10 platform. Unless specified
otherwise, the objective for all the problems is taken as minimizing the
number of workers. The instances are described in the following
subsections. Their results are discussed in Section~\ref{sec:summarysection}.

\subsection{The Software Industry Problem} \label{sec:sipinstance}
The background of this problem was described at the beginning of
Section~\ref{sec:motivation}. For this problem, $\omega=30,\ T=336,\ q=260,\
m_i=m=18$ for $i=1,2,\ldots,q$. This problem is similar to the discontinuous
tour scheduling problem of \cite{stolletz2010operational}. There is no demand
during the periods 10pm to 8am on all days. For an instance of this problem,
the total demand is 1258 worker-hours (see P1 in Table~\ref{tab:results336}), and the
number of workers required over TPs varies from 0 to 25 (see
Fig.~\ref{fig:Callcenterdemand}) with an average of 6.4 workers per TP. The
problem is solved using FX260 shift patterns without imposing any restriction
on the start times of shift patterns. Additional restrictions imposed on the
problem are: (i) at most 5 shifts per worker and (ii) at least 12 hours gap
between any two successive shifts of any worker. For simplicity, we have not
imposed the day-offs to be consecutive. The stage~1 problem was solved in 11
seconds with an optimum objective value of 220 shift schedules. Stage~2
problem was solved in 77 seconds with the optimum objective value of 44
workers. Since this objective value is equal to the lower bound
($44=\lceil \frac{220}{5} \rceil$), the solution is optimal for the problem.

\subsection{Airport Check-in Counter Requirement Problem} \label{sec:airportrequirements}
In this problem, we consider agent requirements to man the check-in counters.
Airline departures over a season (spanning about 6 months) are planned in
advance based on weekly roster. The number of counters allocated to each
airline varies over time depending on the departures of that airline. In
\cite{lalitamm2020}, these requirements were worked out for various airlines'
schedules from a major international airport in India. The weekly departures
of an airline gives rise to agent requirements over the week. Taking domestic
and international departures separately of three airlines, coded as JAW, AAW
and BAW, we formed five demand vectors for  JAW-I, JAW-D, BAW-I, BAW-D and
AAW-D. From these, we derived 10 instances by combining them with the three
shift pattern types FX260, FL15 and FL135, objective function type and the
type of constraint on the worker load. These 10 instances correspond to P2 to
P15 in Table~\ref{tab:results336}. For example, P2 instance is formed by
taking the departures of BAW-I, FX260 shift patterns, worker-load constraint
as the maximum number of shifts per worker in the week and the objective
function as the number of workers. For details of other instances, see the
note under Table~\ref{tab:results336}. P7 to P10 correspond to the instance but
solved under different constraints and different objective functions (see
Fig~\ref{fig:costsandtypes}).  The demands vary from 608 to 3752 agent-hours
(see Table~\ref{tab:results336}). Unlike the discontinuous tour scheduling
problem considered in \cite{stolletz2010operational}, P2 to P15 have
continuous requirements over the planning horizon.

We shall discuss the results of our approach of solving two instances - P4 and P5
with demands 3736 and 1467 agent-hours respectively.
Assuming that the agents can start their shifts at the beginning of every half hour, the
requirements were computed based on 30-minute TPs for weekly planned
departures. Fig~\ref{fig:w9} presents the demand patterns for the two problems.
\begin{figure}[h]
  \centering
  \includegraphics[width=11cm,height=6cm]{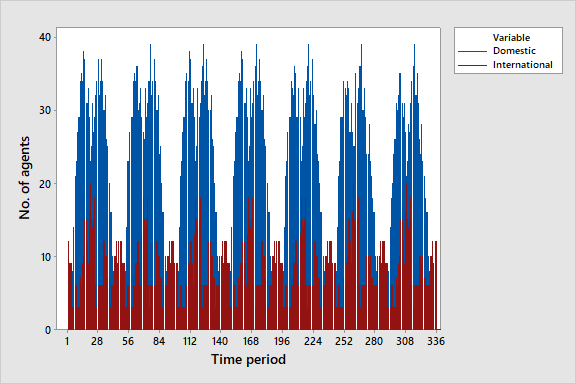}
  \caption{Agent requirements for JAW domestic and international departures}\label{fig:w9}
\end{figure}

Note that \cite{stolletz2010operational} starts with 330 shift schedules (per
day) and uses them in his model to obtain staff schedules over the entire
planning horizon. For problems with FX260, $\omega=30$ and $T=336$, there are
$12480(=260\times 48)$ shift schedules per day. \cite{brunner2014stabilized}
observed that with one flexible lunch break, the number of shift schedules
(per day) rises to 2690 and point out that they were not able to solve the
problem with their model using standard MIP software. The major difference
between the two problems (P4 and P5) and the discontinuous tour scheduling problem considered
in \cite{stolletz2010operational} is the continuous requirement of agents
over the planning horizon. With FX260 patterns and continuous requirements
over one week cyclic planning horizon, the number of possible shifts combinations per worker is approximately
$10^{19}$. The results of solving the two problems, P4 and P5, are summarized below.

\noindent \textbf{Instance P5}\\
The total demand for this problem is 1467 agent-hours. The stage~1 model
for this problem produced a solution with 227 shift schedules in 59 seconds with a lower bound\footnote{The best lower bound produced by the solver during the exucution.}
of 224 on the objective function. The best objective value remained at 227 even after
five minutes CPU running time. We aborted the solver and took
the solution with 227 shift schedules and solved stage~2 problem. This
produced an optimal solution in 78 seconds with 46 agents.
Applying Theorem~\ref{thm2} on the lower bound 224, the number of
agents must be at least 45. Therefore, solution obtained for this problem
with 46 agents is at least 97.7\% optimal.

\noindent\paragraph*{Instance P4} ~ \label{InstanceP4}\\
Recall the discussion about this problem under the introduction of split
technique. While solving stage~2 model for this problem, the solver aborted
the solution process reporting insufficient memory. We observed this
phenomenon whenever we tried to solve stage~2 problem with huge demand. The
reason is that high demand requires large number of shift schedules which in
turn results in large number of overlapping shift schedule pairs. As a
result, the number of constraints under (\ref{con:stage2restperiod})
increases dramatically (for this problem the number of constraints is 5054445
and the number of variables is 87041). Applying the split technique, this
problem was solved in less than 7 minutes
and the optimality gap is at most 6\%.

\noindent \textbf{Instances with cost objective and work load constraints}\\
Since there are shifts with short lengths, we solved stage~2 problem once
with the constraint on the number of shifts per week per worker (maximum 5
shifts) and once with the constraint on maximum number of worker-hours per
week per worker (maximum 50 hours). Again, with respect to objective
function, we have two options, number of workers and cost. Thus, we have four
combinations which are represented by instances P7 to P10 (see
Fig~\ref{fig:costsandtypes}).  For simplicity, cost is taken as a function of
shift duration alone. We took the costs as shown in
Fig~\ref{fig:costsandtypes}.
\begin{figure}[h]
  \centering
  \includegraphics[scale=.5]{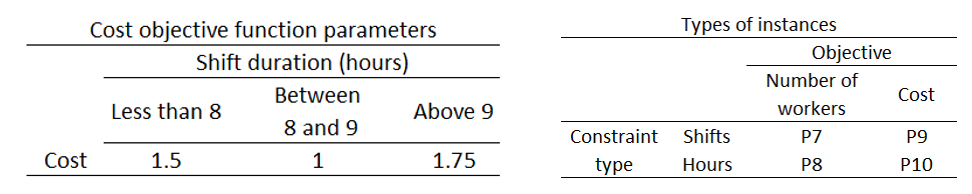}
  \caption{Parameters of instances P7 to P10}\label{fig:costsandtypes}
\end{figure}

\subsection{Call Center Data}
We have data on number of agents worked, hour-wise, for 36 weeks from a call
center. Like in check-in counters problem, the requirement of agents is round
the clock. We combined two weeks (14 days) data to form an instance. For
simplicity, we treat the hourly requirements as requirements for 30 minute
periods. Data are taken from two different streams with total demand varying
from large to very large. There are six instances, P16 to P21 in
Table~\ref{tab:results336}. The variation in the demand pattern is shown in
Figure~\ref{fig:ccdemands}. The total demands vary from 2155 to 5615
agent-hours. As demands are high, all the six instances had to be solved
using split technique. The results are summarized in
Table~\ref{tab:results336}. The solutions to these instances demonstrate the
efficacy of the split technique. Since P21 took a long time (10 minutes and
18 seconds), we solved this problem again (P22) with
$\mathbf{R}=3\mathbf{R}^1+\mathbf{R}^2$ and $\mathbf{R}^1=\lfloor
\frac{\mathbf{R}}{4} \rfloor$ (see Section~\ref{sec:splittechnique}). As a
result, the problem could be solved in less than 5 minutes.
\begin{figure}[h]
  \centering
  \includegraphics[width=11cm,height=6cm]{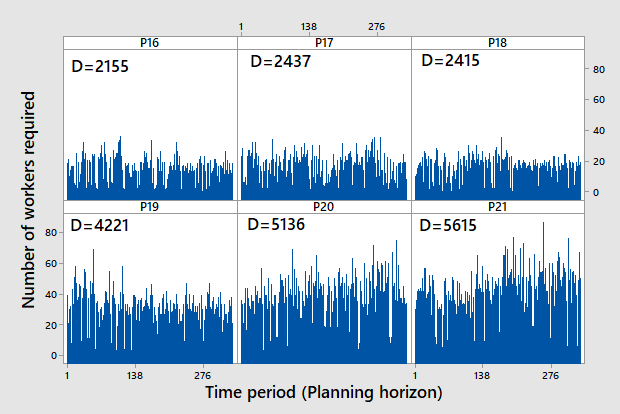}
  \caption{Agent requirements for call center problems}\label{fig:ccdemands}
\end{figure}

\subsection{Instances for ISTSP}
In this section we describe the data for instances on ISTSP which requires task
scheduling as well. We have one live data set from emergency medical services (108
service). For the other instances, P23 to P40 of Table~\ref{tab:results672}, data are
simulated following \cite{volland2017column}. All instances under this case are solved
assuming FX29 patterns.

\noindent \textbf{Medical emergency data (P41)}\\
We have historical data on the 108 service pertaining to a province in
Andhra Pradesh, India. The service brings patients needing emergency medical care to a hospital.
The most commonly reported emergencies (about 70\% of the cases) are related to
pregnancy, acute abdomen, trauma (vehicular), fevers (infections) and cardiac/cardio
vascular issues. Of these, pregnancy cases alone accounted for 23\%. Therefore,
we took data (number of patients arriving in every 15 minutes) on pregnancy cases for one
week (7 consecutive days) of a month. There were 588 such cases. We took the duration of
redressal of these cases (tasks)
to the nearest 15 minutes, and used the seriousness of the cases to set the start windows
and the precedence relations. We took the earliest start time of a task as the arrival TP of the patient, and set the
latest start time  based on the seriousness of the case. The resulting instance has the following characteristics:
$K=588,\ T=672,\ \omega=15$, maximum width of start window of task that is not involved
in precedence relationships is $4\omega$, total demand is 1151 worker-hours ($\sum_{j=1}^{672}R_j=4603$);
number of tasks involved in precedence relationships is 116 with a total of  113 precedence relationships.
The demand pattern is shown in Fig.~\ref{fig:pregnancy}.
\begin{figure}[h]
  \centering
  \includegraphics[scale=.5]{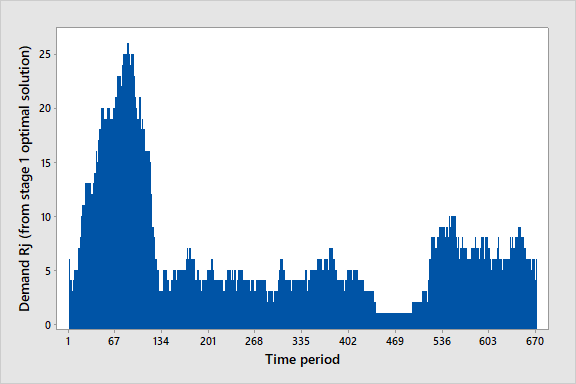}
  \caption{Staff requirements for medical emergency problem}\label{fig:pregnancy}
\end{figure}

The stage~1 model for this problem produced a near optimal solution with objective value of 118 shift schedules
in 119 seconds with a lower bound of 116. We terminated stage~1 at this time and
solved stage~2 problem which produced an optimal solution in 235 seconds with optimum objective value of
37 workers. Applying Theorem~\ref{thm2} on the lower bound 116, the minimum
number of workers is at least 24. In order to find a better lower bound, we took the stage~1 objective
function as $\max_j R_j$ and minimized it. The optimum objective value for this was 26. Therefore,
the number of workers for this problem cannot be less than 26. We then solved stage~1 problem once
again with the original objective function but this time by adding an additional constraint
$\max_j R_j \leq 26$. This resulted in the same objective value and lower bound as before (118/116) but the solution
was different. The resulting solution was used to solve stage~2 problem, and that produced a global
optimum objective value of 36 workers. Thus, the final solution was at least 62\% ($=100 - \frac{36-26}{26}\times 100$)
optimal.

\noindent \textbf{Simulated data}\\
We simulated data for ISTSP following the procedure described in Section~5.2.2 of
\cite{volland2017column}. Under this procedure, three types of tasks (day long, peak and
precedence) and three problem sizes (small (600 hours), medium (1000 hours) and large
(1400 hours)) are considered. For each size, three different distributions of task types
(S1, S2 and S3) were used. The parameters for simulation are summarized in
Table~\ref{fig:Vdist}. Thus, there are nine scenarios under this situation. We simulated
nine instances, P23 to P31 of Table~\ref{tab:results672}, following the procedure for
these nine parameter settings. The corresponding instances from \cite{volland2017column}
are listed in Table~\ref{tab:lmvsvf}. We ignored the additional instances considered by
\cite{volland2017column} (presented in Table~5 of their paper) because those instances
are more restricted (either shift patterns are limited to two or start window lengths are
reduced by 50\%). We simulated the nine instances using the same shift patterns used in
\cite{volland2017column}, namely FL29 shift patterns.
\begin{table}[h]
  \centering
    \caption{Parameters for simulation of instances for ISTSP}\label{fig:Vdist}
  \includegraphics[scale=.45]{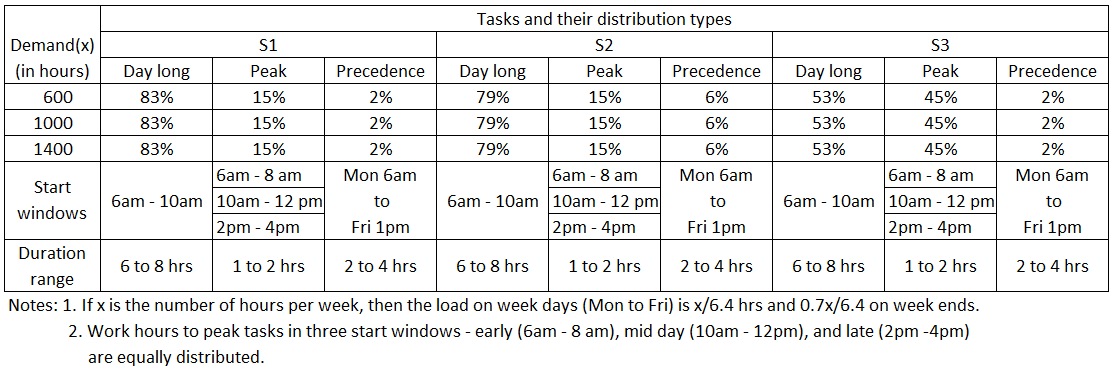}
\end{table}
Additionally, we created nine more instances by considering three more
distributions for the three types of tasks, say S4, S5 and S6. These
distributions are $(81,17,2)$, $(79,17,4)$ and $(72,16,12)$. These are the
resulting distributions if we apply the S1, S2, S3 distributions to number of
tasks instead of applying them to number of hours. These additional 9
instances are P32 to P40 in Table~\ref{tab:results672}. Besides the
differences in the distributions, one major difference between the two sets,
P23 to P31 and P32 to P40, is that in the latter the task start time windows
of all tasks have been chosen uniformly throughout the days. However, we have
not changed the characteristics of the start window widths and task
durations.

\section{Summary of Experimental Results} \label{sec:summarysection}
In this section we shall present the results of our numerical experiments. We have solved
39 problem instances and the results are summarised in Tables \ref{tab:results336} and
\ref{tab:results672}. Table~\ref{tab:results336} presents the results of problems with
fixed demand vector where task scheduling is not required. These problems are similar to
the ones considered in \cite{stolletz2010operational} and \cite{brunner2014stabilized}.
Table~\ref{tab:results672} presents the results for problems with task scheduling
requirements involving precedence relationships. These problems are similar to the ones
considered in \cite{volland2017column}. The parameters affecting the complexity of ISTSP
are: (i)~the length of planning horizon $T$, (ii)~number of tasks, $K$, (iii)~demand and
its pattern ($d_i$s, $r_{ik}$s), (iv)~number of precedence relationships and (v) number
of shift patterns. The range of these parameters in our instances are such that the
results can be compared with the results of the respective papers mentioned above.

To assess the merit of any solution, we consider four parameters: the total
demand, solution time, optimality metric and utilization metric. For problems
where task scheduling is involved, one should also look at the number of tasks
involved in the precedence relationships and the number of precedence
relations. Total demand, expressed as total number of worker-hours required,
is equal to $(\sum_jR_j)\omega/60$. For any solution with objective value
$O_s$ and lower bound $O_L$, the percentage optimality gap is at most
$\frac{O_S-O_L}{O_L}\times 100$. Therefore, we take $\mu = 100
-\frac{O_S-O_L}{O_L}\times 100$ as the measure of optimality. Utilization
metric is taken as 100 times the ratio of total demand to total supply.
Stage~1 model plays a crucial role in our solution approach. We shall first
discuss the results with respect to stage~1 problems.

\subsection{Results for stage 1 model}
In order to apply split technique for large demands in the case of ISTSP,
solving stage~l model efficiently is crucial (see Remark~\ref{rem:readyRvector}).
Fortunately, our experiments show that stage~1 model is solved very efficiently
despite the fact that it is more complex in the case of problems involving task scheduling
with precedence relationships compared to those for which the demand vector
is an input. Fig.\ref{fig:Stage1timeVsdemandAndOptimality} presents the
stage 1 model performance. All solutions are at least 96\% optimal (65\% are 100\% optimal),
and found in less than two minutes (with one exception which took 228 seconds). Average demand
is 2117 worker-hours. It should be noted that the high demand instances took smaller times
(see tables \ref{tab:results336} and \ref{tab:results672}).
\begin{figure}[h]
  \centering
  \includegraphics[scale=.65]{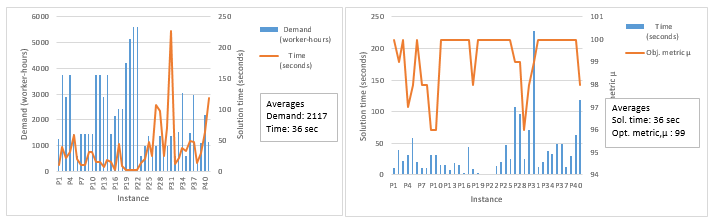}
  \caption{Performance of stage 1 model}\label{fig:Stage1timeVsdemandAndOptimality}
\end{figure}

\subsection{Results of problem instances with given demand vector}
Instances of P1 to P22 are under this category. For each of these instances,
$\omega=30,\ T=336$ and the demand varies from 608 to 5615 worker-hours.
For all instances with demand (number of worker-hours) less than 1500, we
could get solution directly. For the other instances, minimum demand is
above 2000. For these instances, the problems had to be solved using the spit
technique (see the discussion under Instance~P4 on
page~\pageref{InstanceP4}). The method used (`Direct' or `Split($\rho$)') is
specified in Table~\ref{tab:results336}. Instance P11 is solved twice with
$\rho=3,4$. In both cases, the solutions are near optimal (95\% and 98\%),
and the solution times are also close (75 and 81 seconds). Similarly, P21 was
solved twice with $\rho=3,\ 4$. Split(3) took 618 seconds and split(4) took 242 seconds.
In both cases, the solutions are at least 99\% optimal. The necessity for
splitting is arising from large demand. To highlight this, the number of
variables and constraints of stage~2 model with the original demand vector
are presented in Table~\ref{tab:results336}. From the table, it can be seen
that for the instances solved with split technique, the number of constraints
ranges from 1.7 millions to 16.8 millions.
\afterpage{%
\begin{table}[h]
  \centering
 \caption{Results for staff scheduling problem instances}\label{tab:results336}
  \includegraphics[width=12cm, height=14cm]{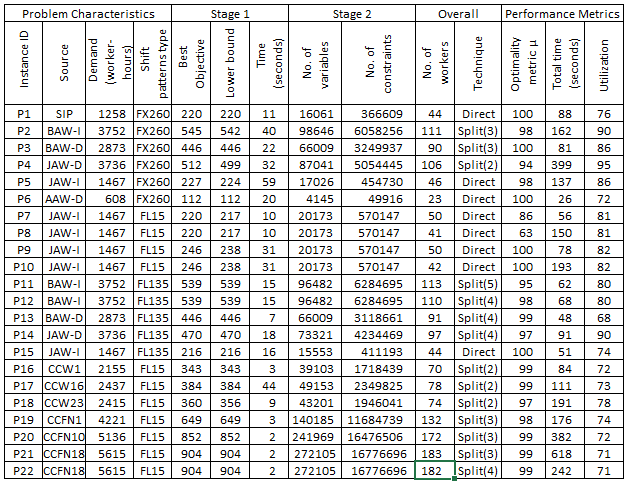}\\
  \begin{tabular}{m{12cm}}
    {\small Note: P1 is SIP, P2 to P15 are check-in counter problems and
           P16 to P21 are call center problems. For all problems, P1 to P21,
           with the exception of P8 and P10,
           the worker load constraint is on the number of shifts, that is, each
           worker is assigned a maximum of 5 shifts; for P8 and P10, it is on
           the number of hours, a maximum of 50 hours per week. Similarly, for all patterns
           other than P9 and P10, the objective function is number of workers, and
           for P9 and P10, it is the cost. P11 and
           P12 are same instance but solved differently. Likewise, P20 and
           P21 are same instance but solved differently. The columns under
           Stage~2 present the size of the problem for the stage~2 problem
           with the original demand vector.
    }
  \end{tabular}
\end{table}
\clearpage
}
The performance metrics of two stage method (with split technique where
needed) as applied to instances of P1 to P22 are presented in the last
three columns of Table~\ref{tab:results336} and in Fig.\ref{fig:TSMBSPerformance}.
In all but two of the instances, the optimality was at least 94\%. In one case, P7,
it is 86\% and in the other case, P8, it is 63\%. The optimality metric $\mu$ for P8
is computed using a poor lower bound, namely total demand by the maximum number of
hours that a worker can be assigned (recall that P8 constraints are based on maximum
number of hours and not the number of shifts, see Fig.\ref{fig:costsandtypes}).
The average solution time is 2 minutes 40 seconds and the average utilization is 79\%.
\begin{figure}[h]
  \centering
  \includegraphics[width=11cm]{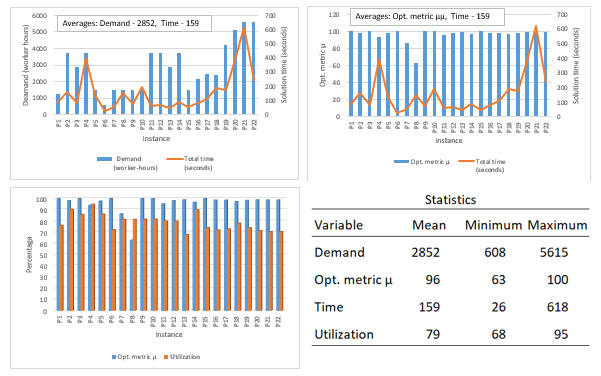}
  \caption{Performance metrics of two stage method for P1 to P17}\label{fig:TSMBSPerformance}
\end{figure}

\cite{stolletz2010operational} reports the solution times for three different cases.
Though our case (continuous demand) is more complex, a comparison is presented
in Fig.\ref{fig:stoltsm} with respect to  solution times.
\begin{figure}[h]    
  \centering
  \includegraphics[scale=.5]{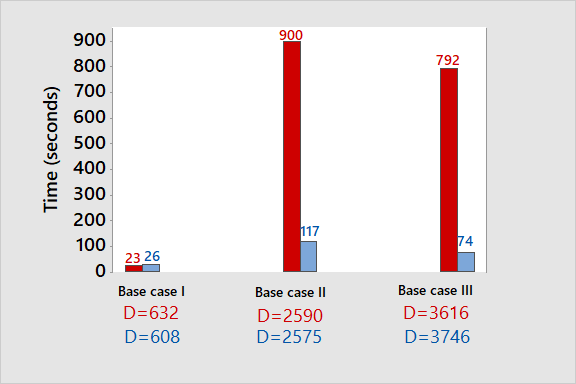}
  \caption{A comparison of solution times. The demands (D) in red are the figures taken from \cite{stolletz2010operational} and the demands in blue are simulated figures for this article.}\label{fig:stoltsm}
\end{figure}

\subsection{Results of ISTSP problem instances}
Instances P23 to P41 are under this category. Task scheduling is
a part of the problem. Results are presented in Table~\ref{tab:results672}.
For these problems, $\omega=15$ and $T=672$, number of tasks $K$ varies from
100 to 588, and the demand varies from 450 to  3032 with an average of
1266 worker-hours. Instances P23 to P40 are simulated, and P41 is based on
a live problem. All simulated instances with the exception of P30 have been
solved to optimality by the two stage method (without the need for
split technique). The solution to P30 is at least
95\% optimal. Utilization in the solutions varied from 75\% to 98\% with an
average of 85\%. Solution times varied from 13 to 427 seconds with an average
of 111 seconds. The solution for the instance with live data (P41) is
at least 58\% optimal but the utilization is 98\%. The demand for this
problem is 1151 worker-hours and it took 354 seconds to solve.
Fig.\ref{fig:VFOverallPerformance} presents the performance of
two stage method.
\begin{table}   
  \centering
 \caption{Results for ISTSP instances}\label{tab:results672}
  \includegraphics[width=12cm, height=12cm]{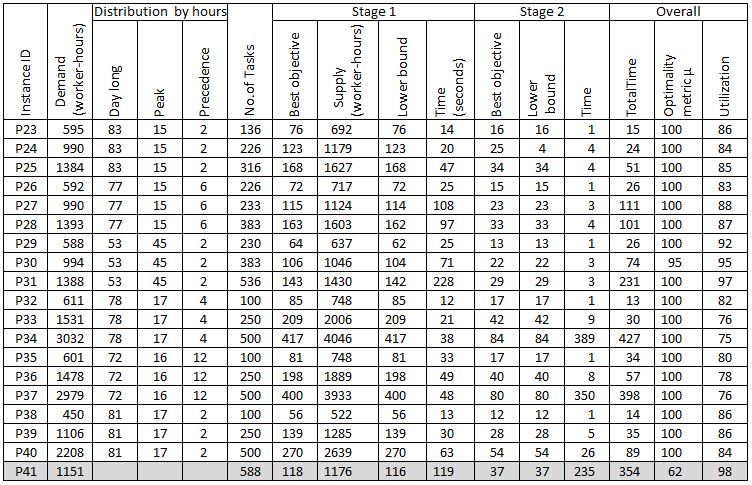}
 \end{table}
\begin{figure}[h]   
  \centering
  \includegraphics[width=11cm]{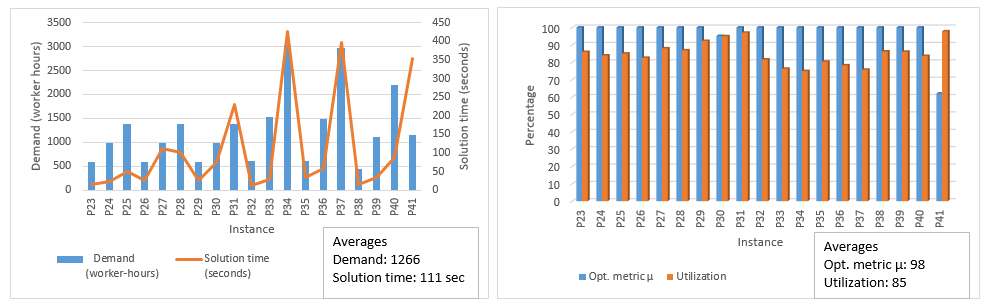}
  \caption{Performance metrics of two stage method for P23 to P41}\label{fig:VFOverallPerformance}
\end{figure}
We shall compare the performance of the two stage method with that of
\cite{volland2017column}. For this, we use the results of instances P23 to P31. Since we
do not have the data used in \cite{volland2017column}, we use the simulation approach.
Recall that instances P23 to P31 are simulated following the procedure stated in
\cite{volland2017column}. It must be pointed out that the comparison is not based
on exact instances but on similar instances. Table~\ref{tab:lmvsvf} presents the
one-to-one correspondence between the two sets of problem instances along with respective
solution times. The solution times for \cite{volland2017column} are taken from their
article. Both methods produced optimal solutions for all the nine instances. The last
column of the table presents the reduction percentages in the solution times. The
solution times are also shown in Fig.\ref{Fig:TsmVsVfTimewise}
\begin{table}[h]   
  \centering
 {\small
    \caption{Comparison with VF w.r. to time wise performance}\label{tab:lmvsvf}
  \begin{tabular}{cccccccc} \hline
         &      & \multicolumn{2}{c}{Correspondence} &  & \multicolumn{2}{c}{Time (seconds)} & Reduction \\ \cline{3-4} \cline{6-7}
Size	 &	Demand	 &	VF	 &	TSM	 &  &	$t_{VF}$	 &	$t_{TSM}$	 &	Percent	 \\ \hline
	 &	595	 &	SMA-S1-LW-FL	 &	P23	 &  &	240	 &	15	 &	93	 \\
Small	 &	592	 &	SMA-S2-LW-FL	 &	P26	 &  &	360	 &	24	 &	93	 \\
	 &	588	 &	SMA-S3-LW-FL	 &	P29	 &  &	900	 &	51	 &	94	 \\
	 &	990	 &	MED-S1-LW-FL	 &	P24	&  &	10800	 &	26	 &	99	 \\
Medium	 &	990	 &	MED-S2-LW-FL	 &	P27	 & &	600	 &	111	 &	81	 \\
	 &	994	 &	MED-S3-LW-FL	 &	P30	 &	& 3180	 &	101	 &	96	 \\
	 &	1384	 &	LAR-S1-LW-FL	 &	P19	 & &	10800	 &	26	 &	99	 \\
Large	 &	1393	 &	LAR-S2-LW-FL	 &	P28	 & &	300	 &	74	 &	75	 \\
	 &	1388	 &	LAR-S3-LW-FL	 &	P31	 &	 & 2760	 &	231	 &	99	 \\  \hline
\multicolumn{8}{m{11.5cm}}{\scriptsize Note: $t_{total}$ is the total time extracted Table~5
            of \cite{volland2017column}; $t_{TSM}$ is the total solution time by two stage method
            (TSM) taken from Table~4. Comparison is made based on similar but not the same instances.}
   \end{tabular}
   }
\end{table}
   \begin{figure}
  \centering
  \includegraphics[scale=.75]{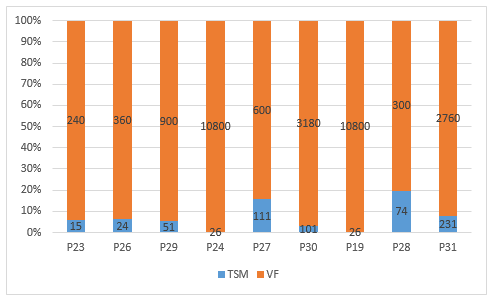}    
  \caption{Comparison of solution times of two stage method (TSM) with
      \cite{volland2017column} (VF) method.}\label{Fig:TsmVsVfTimewise}
\end{figure}

\section{Conclusion} \label{sec:conclusion}
In this article, we considered the integrated staff and task scheduling
problem. The problem is hard to solve even for a predetermined task
schedule. Several authors have considered the problem and proposed column
generation methods to solve. In this article, we proposed a two stage
approach to the problem and introduced the split technique to handle problems
with large demands. We have demonstrated the efficacy of the two stage method
with split technique through a number of numerical experiments in
reducing solution times dramatically. In the existing literature, solution methods
are assessed at different demand sizes such as small, medium and large. Through the split
technique introduced in this article, we are able to handle problems with
large demands efficiently. This raises a question that whether demand size
has any influence on the complexity of the problem. This
point needs to be explored theoretically. Another direction for future research is
extending the methods introduced in this article to multi-skill personnel
staff scheduling problems.

\bibliography{refs}

\begin{thebibliography}{29}
\newcommand{\enquote}[1]{``#1''}
\expandafter\ifx\csname natexlab\endcsname\relax\def\natexlab#1{#1}\fi

\bibitem[\protect\citeauthoryear{Alfares}{Alfares}{2004}]{alfares2004survey}
\textsc{Alfares, H.~K.} (2004): \enquote{Survey, categorization, and comparison
  of recent tour scheduling literature,} \emph{Annals of Operations Research},
  127, 145--175.

\bibitem[\protect\citeauthoryear{Alfares and Bailey}{Alfares and
  Bailey}{1997}]{alfares1997integrated}
\textsc{Alfares, H.~K. and J.~E. Bailey} (1997): \enquote{Integrated project
  task and manpower scheduling,} \emph{IIE transactions}, 29, 711--717.

\bibitem[\protect\citeauthoryear{Aykin}{Aykin}{1996}]{aykin1996optimal}
\textsc{Aykin, T.} (1996): \enquote{Optimal shift scheduling with multiple
  break windows,} \emph{Management Science}, 42, 591--602.

\bibitem[\protect\citeauthoryear{Bailey, Alfares, and Lin}{Bailey
  et~al.}{1995}]{bailey1995optimization}
\textsc{Bailey, J., H.~Alfares, and W.~Y. Lin} (1995): \enquote{Optimization
  and heuristic models to integrate project task and manpower scheduling,}
  \emph{Computers \& Industrial Engineering}, 29, 473--476.

\bibitem[\protect\citeauthoryear{Bassett}{Bassett}{2000}]{bassett2000assigning}
\textsc{Bassett, M.} (2000): \enquote{Assigning projects to optimize the
  utilization of employees' time and expertise,} \emph{Computers \& Chemical
  Engineering}, 24, 1013--1021.

\bibitem[\protect\citeauthoryear{Beli{\"e}n and Demeulemeester}{Beli{\"e}n and
  Demeulemeester}{2008}]{belien2008branch}
\textsc{Beli{\"e}n, J. and E.~Demeulemeester} (2008): \enquote{A
  branch-and-price approach for integrating nurse and surgery scheduling,}
  \emph{European journal of operational research}, 189, 652--668.

\bibitem[\protect\citeauthoryear{Beli{\"e}n, Demeulemeester, De~Bruecker,
  Van~den Bergh, and Cardoen}{Beli{\"e}n et~al.}{2013}]{belien2013integrated}
\textsc{Beli{\"e}n, J., E.~Demeulemeester, P.~De~Bruecker, J.~Van~den Bergh,
  and B.~Cardoen} (2013): \enquote{Integrated staffing and scheduling for an
  aircraft line maintenance problem,} \emph{Computers \& Operations Research},
  40, 1023--1033.

\bibitem[\protect\citeauthoryear{Bellenguez-Morineau and
  N{\'e}ron}{Bellenguez-Morineau and N{\'e}ron}{2007}]{bellenguez2007branch}
\textsc{Bellenguez-Morineau, O. and E.~N{\'e}ron} (2007): \enquote{A
  branch-and-bound method for solving multi-skill project scheduling problem,}
  \emph{RAIRO-operations Research}, 41, 155--170.

\bibitem[\protect\citeauthoryear{Brucker, Qu, and Burke}{Brucker
  et~al.}{2011}]{brucker2011personnel}
\textsc{Brucker, P., R.~Qu, and E.~Burke} (2011): \enquote{Personnel
  scheduling: Models and complexity,} \emph{European Journal of Operational
  Research}, 210, 467--473.

\bibitem[\protect\citeauthoryear{Brunner, Bard, and Kolisch}{Brunner
  et~al.}{2009}]{brunner2009flexible}
\textsc{Brunner, J.~O., J.~F. Bard, and R.~Kolisch} (2009): \enquote{Flexible
  shift scheduling of physicians,} \emph{Health care management science}, 12,
  285--305.

\bibitem[\protect\citeauthoryear{Brunner, Bard, and Kolisch}{Brunner
  et~al.}{2010}]{brunner2010midterm}
---\hspace{-.1pt}---\hspace{-.1pt}--- (2010): \enquote{Midterm scheduling of
  physicians with flexible shifts using branch and price,} \emph{Iie
  Transactions}, 43, 84--109.

\bibitem[\protect\citeauthoryear{Brunner and Stolletz}{Brunner and
  Stolletz}{2014}]{brunner2014stabilized}
\textsc{Brunner, J.~O. and R.~Stolletz} (2014): \enquote{Stabilized branch and
  price with dynamic parameter updating for discontinuous tour scheduling,}
  \emph{Computers \& operations research}, 44, 137--145.

\bibitem[\protect\citeauthoryear{Dantzig}{Dantzig}{1954}]{dantzig1954letter}
\textsc{Dantzig, G.~B.} (1954): \enquote{Letter to the editor—A comment on
  Edie's “Traffic delays at toll booths”,} \emph{Journal of the Operations
  Research Society of America}, 2, 339--341.

\bibitem[\protect\citeauthoryear{Di~Martinelly, Baptiste, and
  Maknoon}{Di~Martinelly et~al.}{2014}]{di2014assessment}
\textsc{Di~Martinelly, C., P.~Baptiste, and M.~Maknoon} (2014): \enquote{An
  assessment of the integration of nurse timetable changes with operating room
  planning and scheduling,} \emph{International Journal of Production
  Research}, 52, 7239--7250.

\bibitem[\protect\citeauthoryear{Ernst, Jiang, Krishnamoorthy, and Sier}{Ernst
  et~al.}{2004}]{DBLP:journals/eor/ErnstJKS04}
\textsc{Ernst, A.~T., H.~Jiang, M.~Krishnamoorthy, and D.~Sier} (2004):
  \enquote{Staff scheduling and rostering: {A} review of applications, methods
  and models,} \emph{European Journal of Operational Research}, 153, 3--27.

\bibitem[\protect\citeauthoryear{Hartmann and Briskorn}{Hartmann and
  Briskorn}{2010}]{hartmann2010survey}
\textsc{Hartmann, S. and D.~Briskorn} (2010): \enquote{A survey of variants and
  extensions of the resource-constrained project scheduling problem,}
  \emph{European Journal of operational research}, 207, 1--14.

\bibitem[\protect\citeauthoryear{Jacobs and Brusco}{Jacobs and
  Brusco}{1996}]{jacobs1996overlapping}
\textsc{Jacobs, L.~W. and M.~J. Brusco} (1996): \enquote{Overlapping start-time
  bands in implicit tour scheduling,} \emph{Management Science}, 42,
  1247--1259.

\bibitem[\protect\citeauthoryear{Jarrah, Bard, and deSilva}{Jarrah
  et~al.}{1994}]{jarrah1994solving}
\textsc{Jarrah, A.~I., J.~F. Bard, and A.~H. deSilva} (1994): \enquote{Solving
  large-scale tour scheduling problems,} \emph{Management Science}, 40,
  1124--1144.

\bibitem[\protect\citeauthoryear{Kim and Mehrotra}{Kim and
  Mehrotra}{2015}]{kim2015two}
\textsc{Kim, K. and S.~Mehrotra} (2015): \enquote{A two-stage stochastic
  integer programming approach to integrated staffing and scheduling with
  application to nurse management,} \emph{Operations Research}, 63, 1431--1451.

\bibitem[\protect\citeauthoryear{Lalita, Manna, and Murthy}{Lalita
  et~al.}{2020}]{lalitamm2020}
\textsc{Lalita, T.~R., D.~K. Manna, and G.~S.~R. Murthy} (2020):
  \enquote{Mathematical Formulations for Large Scale Check-in Counter
  Allocation Problem,} \emph{Journal of Air Transport Management}, to appear,
  ---.

\bibitem[\protect\citeauthoryear{Maenhout and Vanhoucke}{Maenhout and
  Vanhoucke}{2013}]{maenhout2013integrated}
\textsc{Maenhout, B. and M.~Vanhoucke} (2013): \enquote{An integrated nurse
  staffing and scheduling analysis for longer-term nursing staff allocation
  problems,} \emph{Omega}, 41, 485--499.

\bibitem[\protect\citeauthoryear{Maenhout and Vanhoucke}{Maenhout and
  Vanhoucke}{2016}]{maenhout2016exact}
---\hspace{-.1pt}---\hspace{-.1pt}--- (2016): \enquote{An exact algorithm for
  an integrated project staffing problem with a homogeneous workforce,}
  \emph{Journal of Scheduling}, 19, 107--133.

\bibitem[\protect\citeauthoryear{Stolletz}{Stolletz}{2010}]{stolletz2010operational}
\textsc{Stolletz, R.} (2010): \enquote{Operational workforce planning for
  check-in counters at airports,} \emph{Transportation Research Part E:
  Logistics and Transportation Review}, 46, 414--425.

\bibitem[\protect\citeauthoryear{Stolletz and Zamorano}{Stolletz and
  Zamorano}{2014}]{stolletz2014rolling}
\textsc{Stolletz, R. and E.~Zamorano} (2014): \enquote{A rolling planning
  horizon heuristic for scheduling agents with different qualifications,}
  \emph{Transportation Research Part E: Logistics and Transportation Review},
  68, 39--52.

\bibitem[\protect\citeauthoryear{Sungur, {\"O}zg{\"u}ven, and Kariper}{Sungur
  et~al.}{2017}]{sungur2017shift}
\textsc{Sungur, B., C.~{\"O}zg{\"u}ven, and Y.~Kariper} (2017): \enquote{Shift
  scheduling with break windows, ideal break periods, and ideal waiting times,}
  \emph{Flexible Services and Manufacturing Journal}, 29, 203--222.

\bibitem[\protect\citeauthoryear{Thompson}{Thompson}{1995}]{thompson1995improved}
\textsc{Thompson, G.~M.} (1995): \enquote{Improved implicit optimal modeling of
  the labor shift scheduling problem,} \emph{Management Science}, 41, 595--607.

\bibitem[\protect\citeauthoryear{Thompson and Pullman}{Thompson and
  Pullman}{2007}]{thompson2007scheduling}
\textsc{Thompson, G.~M. and M.~E. Pullman} (2007): \enquote{Scheduling
  workforce relief breaks in advance versus in real-time,} \emph{European
  Journal of Operational Research}, 181, 139--155.

\bibitem[\protect\citeauthoryear{Van~den Bergh, Beli{\"e}n, De~Bruecker,
  Demeulemeester, and De~Boeck}{Van~den Bergh et~al.}{2013}]{van2013personnel}
\textsc{Van~den Bergh, J., J.~Beli{\"e}n, P.~De~Bruecker, E.~Demeulemeester,
  and L.~De~Boeck} (2013): \enquote{Personnel scheduling: A literature review,}
  \emph{European journal of operational research}, 226, 367--385.

\bibitem[\protect\citeauthoryear{Volland, F{\"u}gener, and Brunner}{Volland
  et~al.}{2017}]{volland2017column}
\textsc{Volland, J., A.~F{\"u}gener, and J.~O. Brunner} (2017): \enquote{A
  column generation approach for the integrated shift and task scheduling
  problem of logistics assistants in hospitals,} \emph{European Journal of
  Operational Research}, 260, 316--334.

\end{thebibliography}
\bibliographystyle{ecta}

\end{document}